\documentclass{amsart}

\usepackage{graphicx}              
\usepackage{amsmath}               
\usepackage{amsfonts}              
\usepackage{amsthm}                
\usepackage{amssymb}
\usepackage{amscd}
 \usepackage{listings}
\usepackage{times}
\usepackage{comment}
\usepackage{mathtools}
\usepackage{tikz}
\usetikzlibrary{arrows, matrix}
\usepackage{tikz-cd}
\usepackage{calc}
\usetikzlibrary{calc}

\usepackage{stmaryrd}
\usepackage{xcolor}

\theoremstyle{definition}
\newtheorem{thm}{Theorem}[section]
\newtheorem{lem}[thm]{Lemma}

\newtheorem{defn}[thm]{Definition}

\newtheorem{rem}[thm]{Remark}

\newcommand\QQ{{\mathbf Q}}
\newcommand{\RR}{\mathbf{R}}      
\newcommand{\ZZ}{\mathbf{Z}}      
\newcommand\CC{{\mathbf C}}
\newcommand{\Gm}{\mathbf{G}_m}
\newcommand{\kk}{\mathsf{k}}

\newcommand{\underg}{\underline{g}}

\newcommand{\Z}{\mathbf{Z}}
\newcommand{\Q}{\mathbf{Q}}
\newcommand{\R}{\mathbf{R}}

\newcommand{\C}{\mathbf{C}}
\newcommand{\F}{\mathbf{F}}

\newcommand{\spec}{\mathrm{Spec}\, }

\newcommand{\acts}{\curvearrowright}

\newcommand{\Ggr}{\mathbf{G}_{\mathrm{gr}}}

\usepackage{hyperref}
\hypersetup{
    colorlinks,
    citecolor=blue,
    filecolor=blue,
    linkcolor=blue,
    urlcolor=blue
}

\begin{document}

\author{Eric Y. Chen}
\title{Relative Langlands duality of Toric Periods}
\begin{abstract}
The relative Langlands program introduced by Ben-Zvi--Sakellaridis--Venkatesh \cite{BZSV} posits a duality structure exchanging automorphic periods and $L$-functions, which can be encoded by pairs of dual Hamiltonian actions. In \cite{CV}, an extension of the definitions to certain singular spaces was made with the objective of restoring duality in some well-known automorphic integrals. In this companion article we apply the definitions of \textit{loc. cit} to establish duality in the context of affine toric varieties, and study finer structures regarding regularization that are instructive for the general case.
\end{abstract}

\maketitle
\tableofcontents

\section{Introduction}

\subsection{Relative Langlands duality} 
Let $G$ and $\check{G}$ be a pair of Langlands dual reductive groups defined over $\ZZ$ (or a localization thereof). In \cite{BZSV}, a notion of ``relative Langlands duality" is proposed, which posits that the duality between $G$ and $\check{G}$ should extend to a duality between certain Hamiltonian $G$ and $\check{G}$-actions. These Hamiltonian actions index objects and numerical measurements on the two sides of Langlands duality, and their duality expresses the coincidence of these data. 

For simplicity of exposition, we will assume for the rest of the introduction that our Hamiltonian actions are \textit{polarized}, i.e., we index them by $G$-spaces $X$ and $\check{G}$-spaces $\check{X}$; one may recover the underlying Hamiltonian action by taking cotangent bundles.\footnote{Be aware that, in this simplification, our notation is a bit misleading: the duality structure ought to exist between $M = T^*X$ and $\check{M} = T^*\check{X}$, but not between $X$ and $\check{X}$ themselves.} From an arithmetic perspective, one of the most interesting coincidences of numerical measurements takes the following shape. Given a pair of dual Hamiltonian actions $(G,M = T^*X)$ and $(\check{G}, \check{M} = T^*\check{X})$, one expects a pair of equations: for a Hecke eigenform $f_G$ on $G$ and a Hecke eigenform $f_{\check{G}}$ on $\check{G}$,
\begin{equation}\label{eq: motivation1}
    \langle X\text{-Poincar\'e series}, f_G\rangle \sim L(\check{X}, f_G)
\end{equation}
\begin{equation}\label{eq: motivation2}
    \langle \check{X}\text{-Poincar\'e series}, f_{\check{G}}\rangle \sim L(X, f_{\check{G}})
\end{equation}

Summarizing crudely, the left hand sides of these equations are $L^2$-pairings between a certain Poincar\'e series and a chosen automorphic form, recovering various familiar period integrals in the literature, while the right hand sides of these equations are \textit{nonlinear} generalizations of $L$-functions, which recover Langlands $L$-functions in the case when $X$ (or $\check{X}$) is a linear representation. For a more extensive introduction to and motivation of these objects, we refer the reader to the Introduction of \cite{CV}.

To make Equations \eqref{eq: motivation1} and \eqref{eq: motivation2} precise, various regularization procedures were adopted in \textit{loc. cit}; to name a few, one restricts to cusp forms on the left hand side, discards smaller orbits, and formally ``cancels" various $\zeta(1)$'s that occur from central tori on both sides. These tricks should be viewed as a shortcut with the goal of ignoring trivial (but interesting) divergences and to establish a mathematical statement as quickly as possible, but a complete analysis of dual Hamiltonian actions should provide a systematic way to regularize these divergences.

One of the principal goals of this note is to study these finer details that go beyond \textit{weak numerical duality} (see \S \ref{subsect: weakDuality} for a definition) in the simple setting of toric varieties, in which the most salient features are already visible. 

In the rest of the Introduction, we discuss the three main aspects of relative duality which go beyond the \textit{hyperspherical} setting (see \S 3.5 of \cite{BZSV} for a definition) that we analyze in the toric case: singularities, regularization, and disconnected stabilizers. 

\subsection{Singular spaces}
In order to make manifest the underlying duality in the Garrett triple product integral \cite{Garrett} and Ginzburg's adjoint integral \cite{Ginzburg}, we argued in \cite{CV} that it is necessary to consider the case when $X$ and $\check{X}$ are singular. In this note we study another class of singular spaces: affine toric varieties with normal singularities. This is a class of actions that is closed under weak numerical duality according to one of our main theorems (see Theorem \ref{thm: main1}), and the lesson extracted is the same as that of \cite{CV}:
\begin{quote}
    integrality condition on $X(F_v)$ $\leftrightarrow$ weights of functions appearing on $\check{X}$, and
\end{quote}
\begin{quote}
    integrality condition on $\check{X}(F_v)$ $\leftrightarrow$ weights of functions appearing on $X$.
\end{quote}
As is often the case when studying toric varieties, the weak numerical duality between these toric varieties can be understood combinatorially. We define the notion of \textit{toric dual} toric varieties (see Definition \ref{defn: toricDual}), which seems to us a natural combinatorial definition although as far as we are aware, has not appeared in the toric varieties literature.

\subsection{Regularization}
As is well-known in the theory of integral representations of $L$-functions, one often discards all but one \textit{main orbit} in $X$ when calculating the $L^2$-pairing on the left hand side of Equation \eqref{eq: motivation1}; indeed, contributions from smaller orbits are often proportional to constant term integrals and thus may be ignored in the analysis of periods of cusp forms. 

On the other hand, when $X$ is placed on the spectral side, i.e. when we consider the right hand side of Equation \eqref{eq: motivation2}, it is in fact the smallest orbit that plays the principal role of determining whether the expression $L(X, f_{\check{G}})$ vanishes or not. Roughly speaking, weak numerical duality is a matching between
\begin{quote}
    main orbit automorphic $X$-period $\leftrightarrow$ smallest orbit spectral $\check{X}$-period, and 
\end{quote}
\begin{quote}
    main orbit automorphic $\check{X}$-period $\leftrightarrow$ smallest orbit spectral $X$-period.
\end{quote}
One expects that this matching extends to a inclusion-reversing correspondence between all orbits of $X$ and $\check{X}$, although the resulting automorphic and spectral contributions from these secondary orbits require regularization to define properly. 

While the notion of cusp forms and constant terms are absent in the toric case, we observe in \S \ref{sect: regularize} that the correspondence between orbits alluded to above is present and explicitly computable; in particular, the matching of ``smaller" orbit automorphic contributions from $X$ can be matched with ``bigger" orbit spectral contributions from $\check{X}$, and vice versa, following the regularization scheme of Definition \ref{def: regularizeAut} and Definition \ref{def: regularizeSpec}.

\subsection{Disconnected stabilizers}

In the definition of hypersphericity one stipulates that the stabilizer of a generic point is connected. Again in the toric case, we consider when the stabilizer is a finite subgroup scheme $\mu \subset T$; in other words, the action of $T$ factors through the quotient torus $T/\mu$. It is natural to construct toric Deligne--Mumford stacks from this setup, and in Section \ref{appendix: DM} we formalize the heuristic that
\begin{quote}
    disconnected stabilizers $\leftrightarrow$ Deligne--Mumford stabilizer on the dual
\end{quote}
by defining an automorphic and spectral period for these toric Deligne--Mumford stacks. For instance, the following pair of actions 
$$(T = \Gm, \mathbf{A}^1) \text{ and } (\check{T} = \Gm, [\mathbf{A}^1/\ZZ/n\ZZ])$$
where $T$ acts on $\mathbf{A}^1$ \textit{with weight $n \in \mathbf{N}$} will be a weakly numerically dual pair in our sense. This was already foreseen from Coulomb branch calculations of the weight $n$ action of $\Gm$ on $\mathbf{A}^1$ (see Theorem 4.1 of \cite{BFN}).

\subsection{Notation} \label{notation} 
\subsubsection{Tori}
 We denote by $T$ a split torus defined over $\ZZ$. Its character lattice (resp. cocharacter lattice) will be denoted by $X^*(T)$ (resp. $X_*(T)$). The Langlands dual torus of $T$ will be denoted $\check{T}$, which is a split torus defined over $\ZZ$ with character lattice $X_*(T)$ and cocharacter lattice $X^*(T)$.

\subsubsection{Coefficient fields} \label{Fandk}
$\F$ will be used for an ``automorphic'' coefficient field; it will always
be a finite field of size a prime power $q$. The letter
$\kk$ will denote a choice of ``Galois-side'' coefficient field: an algebraically closed field of characteristic 0. It will be convenient to fix, once and for all, an isomorphism
$\kk \simeq \C$ so we may freely move between $\kk$-valued and $\CC$-valued notions without explicit comment.

 \subsubsection{Group actions} \label{Xdef}
Our convention for group actions (in accordance with that of \cite{CV}) is as follows:
 \begin{quote}
Group actions on spaces on the right, and group actions on functions, forms etc.  are on the left,
derived from geometric actions by pullback.\footnote{Since we are considering only tori, one can be cavalier about this convention upon a first pass; however, we defer to \S A.1 of \cite{CV} for the discussion of the consequences of such conventions when comparing with ``usual" $L$-functions.}
\end{quote}

  Let $X, \check{X}$ be affine $\ZZ$-varieties, 
 admitting actions of 
  $T \times \Ggr$
and $\check{T} \times \Ggr$ respectively.   Here, $\Ggr$ is simply another name for $\Gm$ and is in fact identified with a one-parameter subgroup of $T$ (or $\check{T}$) via a choice of cocharacter, but it plays a distinguished role and to avoid confusion we label it differently. The most important
function of the $\Ggr$ action is that it will govern, in general, where the $L$-function is to be evaluated,
when $X$ or $\check{X}$ is on the spectral side. 
An  {\em eigenmeasure} is a
differential form of top degree on the smooth locus $X^{\circ}$
which is an eigenvector under the translation action of both $T$ and $\Ggr$: 
\begin{equation} \label{etadef}  (g, \lambda)^* \omega = \eta(g, \lambda) \omega = \eta(g)\lambda^\varepsilon \omega,\end{equation} 
for a character $\eta: T \times \Ggr \rightarrow \Gm$ and an integer $\varepsilon \in \ZZ$. 

Our primary interest is in the case when $X, \check{X}$ are {\em conical}, in the following sense.
\begin{defn}\label{defn: conical}
Let $k$ be a field, and let $X$ be an affine $k$-variety with $\Gm$-action. We say that $X$ is \textit{conical} if the coordinate ring $k[X]$ has only nonnegative $\Gm$-weights, and the 0th graded piece is isomorphic to $k$.
(In particular, such an $X$ has a unique $\Gm$-fixed point, usually to be denoted by $0$). 
\end{defn}

For our purposes, $k$ will either be $\F$ (the automorphic coefficient field) or $\kk$ (the Galois coefficient field). A $T \times \Ggr$-variety $X$ is \textit{conical} if it is conical in the above sense with $\Gm = \Ggr$.

 \subsubsection{Curves}
 
Let $\Sigma$ be a curve of genus $g$ over $\F$ and with function field $F$. 
We introduce the letter $\Delta$ for the discriminant of $\Sigma$, that is to say:
$$ \Delta :=  q^{2g-2}.$$
We will write $\zeta(s)$ for the $\zeta$-function of $\Sigma$, i.e.
$\zeta(s) = \prod_{v} (1-q_v^{-s})^{-1}$ where the product ranges over places $v$
with residue field of size $q_v$.

We denote by $\mathbb{A}$ the adele ring of $F$
and by $\mathcal{O} \subset \mathbb{A}$ the maximal compact subring. 
Write \begin{equation} \label{bracketnotation} [T] := T_F \backslash T_{\mathbb{A}} / T_{\mathcal{O}}\end{equation} 
the adelic quotient, equivalently, the
set of isomorphism classes of  ($\F$-rational) $T$-bundles over $\Sigma$.

 An ``unramified automorphic form'' on $T$
will be, by definition, a 
\textit{unitary} character 
$$\chi: [T] \rightarrow \kk^\times$$

\subsubsection{Local notation}
We write $v$ for a place of $F$ and $F_v, \mathcal{O}_v, \varpi_v$ for the completion of $F$, the local ring of integers, and the uniformizer, at $v$. We write $q_v$ for the size of the local residue field at $v$. The normalized local valuation will be denoted by $x \mapsto |x|_v$ and
its product over all places gives the adelic valuation $\mathbb{A}^{\times} \rightarrow \R_{>0}$, often
denoted simply $x \mapsto |x|$. For $\lambda$ a cocharacter of $T$, we often write $\varpi^\lambda_v := \lambda(\varpi_v) \in T(F_v)$. These elements form a complete set of $T(\mathcal{O}_v)$-orbit representatives on $T(F_v)$.

\subsubsection{Galois parameters}

We will write
$$\Gamma \mbox{ or } \Gamma_F = \mbox{Weil group of $\Sigma$}$$
for the  unramified global Weil group of $\Sigma$, i.e. the preimage of integer powers of Frobenius inside the {\'e}tale fundamental
group (equivalently: everywhere unramified Galois group) of $\Sigma$. 

\subsubsection{Spin structures} 

Let $\psi: \mathbb{A}/F \rightarrow \CC^{\times}$
be an additive ``Whittaker" character 
 whose conductor at each place $v$ is even.
 That is to say, there exists, for each place $v$,
an {\em even} integer $2m_v$, with the property that 
$\psi$ is trivial on $\varpi_v^{-2m_v} \mathcal{O}_v$
but not on $\varpi_v^{-2m_v-1} \mathcal{O}_v$. 
Such a character exists by a theorem of Hecke (see \cite{Hecke} Satz 176).\footnote{ Let $K^{1/2} = K^{1/2}_\Sigma$ be a choice of $\F$-rational spin structure;
 it exists, by a theorem of Hecke, and we fix a rational section $\nu$ of $K^{1/2}$ which determines also a rational 1-form $\omega = \nu^{\otimes 2}$.
 Let $2m_v$ be the vanishing order of $\omega$ at a place $v$.

 We take $\psi$ to be the character given by 
 determined by $\omega$, i.e. locally sending $f$ to $\mathrm{Res}(f \omega)$.} 

We write
 \begin{equation} \label{partialdef} \partial^{1/2} = (\varpi_v^{m_v}) \in \mathbb{A}^\times, \mbox{ so that } |\partial| =q^{-(2g-2)} = \Delta^{-1}.\end{equation}
While the Whittaker character will not be relevant since we only consider automorphic forms on tori, the spin structure that we are inherently fixing will be.

\subsubsection{Measures} \label{groupmeasures}
For $G$ a reductive group over $\F$, we shall 
  normalize the Haar measure on $G(\mathbb{A})$ in such a way that 
it assigns mass $1$ to the standard maximal compact subgroup
$\prod_{v} G(\mathcal{O}_v)$.

\subsubsection{Normalization of class field theory} 
Our normalization of local class field theory
associates the modulus character $x \mapsto |x|$
of a local field $F$ with the cyclotomic character of its Galois group, that is to say, sending a geometric Frobenius element
to $q^{-1}$ where $q$ is the size of the residue field; or, said differently, the reciprocity map
of class field theory
sends a uniformizing element to geometric Frobenius. 
In what follows,  Frobenius means geometric Frobenius.

\subsection{Acknowledgements}
We would like to thank Akshay Venkatesh for their continual encouragement and their interest in this project, Philippe Michel for pointing out the relationship to the work of Batyrev--Tshinkel on Manin's conjecture for toric varieties, Raphaël Beuzart-Plessis for helpful discussions regarding the integrals of Section \ref{appendix: DM}, and Dimitri Wyss and Sergej Monavari for conveniently teaching a course on toric varieties during the fall semester of 2023. This work was supported by the Swiss National Science Foundation No. 196960.

\section{Combinatorial conventions of toric varieties}

Let $T, \check{T}$ be a pair of (split) Langlands dual tori defined over $\ZZ$. They are defined by canonical isomorphisms
$$X^*(T) \cong X_*(\check{T}) \text{ and } X_*(T) \cong X^*(\check{T})$$
between the character lattice of $T$ with the cocharacter lattice of $\check{T}$, and between the cocharacter lattice of $T$ and the character lattice of $\check{T}$.

\subsection{Cones and fans}
Given a cone $\sigma \subset X_*(T)_\RR$ which is  \textit{strongly convex} (meaning convex and not containing any linear subspaces) and \textit{rational} (meaning that its faces are defined by rationally defined hyperplanes), one associates an affine $T$-toric variety defined over $\Z$:
\begin{equation}\label{eq: affineToricDefn}
    U_\sigma := \spec \Z[\check{\sigma} \cap X^*(T)]
\end{equation}
where $\check{\sigma}$ is the dual cone of $\sigma$, defined by
$$\check{\sigma} := \big\{\chi: \langle \sigma, \chi\rangle \geq 0\big\}$$

Although we will only consider normal affine toric varieties, or equivalently, strongly convex rational cones, it is useful to work with the notion of a fan when discussing orbits inside our toric varieties, so we briefly review the nomenclature of this theory for the reader's convenience. Given a fan $\Sigma$\footnote{By a \textit{fan}, we mean a collection of rational strongly convex cones. The collection satisfies the following properties: 1) the intersection of two cones in the collection is a face of both cones, and 2) all the faces of all the cones in the collection also belong to the collection.} in $X_*(T)_\RR$, one associates a $T$-toric variety $X_\Sigma$ defined over $\Z$ as follows:
\begin{itemize}
    \item For each cone $\sigma$ in $\Sigma$, one has the affine scheme $U_\sigma$ defined by \eqref{eq: affineToricDefn}. 
    \item Suppose $\sigma_1, \sigma_2$ are two cones who share a face $\tau$. Then there is some $\chi \in X^*(T)$ whose hyperplane $H_\chi := \big\{x \in X_*(T): \langle \chi,x\rangle = 0\big\}$ intersects with $\sigma_1,\sigma_2$ exactly at $\tau$:
    $$\tau = H_\chi \cap \sigma_1 = H_\chi \cap \sigma_2,$$
    and we may assume without loss of generality that $\chi$ is positive on the interior of $\sigma_1$. For such a $\chi$, there is a canonical isomorphism 
    $$\spec \ZZ[\check{\sigma}_1 \cap X^*(T)][\chi^{-1}] \cong U_\tau \cong \spec \ZZ[\check{\sigma_2} \cap X^*(T)][\chi]$$
    that allows us to glue $U_{\sigma_1}$ and $U_{\sigma_2}$.
    \item We obtain 
    $$X_\Sigma := \big(\sqcup_{\sigma \in \Sigma} \, U_\sigma\big)/\sim$$
    where the equivalence relation $\sim$ is described in the gluing step above.
\end{itemize}
If $\sigma \subset X_*(T)_\RR$ is a rational strongly convex cone, then we may consider the fan $\Sigma$ consisting of all faces of $\sigma$, and obtain the same toric variety $U_\sigma \cong X_\Sigma$ using either definition. 

The combinatorics of inclusions of cones dictate the geometry of its associated toric variety. In particular, we have the following useful correspondence between $T$-orbits on $X_\Sigma$ and cones in $\Sigma$.

\begin{thm}[Theorem 3.2.6 \cite{CLS}]\label{thm: orbitConeCorresp}
Let $X$ be the toric variety associated to the fan $\Sigma$ in $X_*(T)_\RR$. Then there is a bijective correspondence    
$$\bigg\{\text{Cones in }\Sigma\bigg\} \longleftrightarrow \bigg\{T\text{-orbits in } X\bigg\}$$
$$\sigma \longleftrightarrow O_\sigma = T/S_\sigma$$
where $S_\sigma \subset T$ is the subtorus with cocharacter lattice generated by $\sigma \cap X_*(T)$, with the following properties (we write $\tau \leq \sigma$ if $\tau$ is a cone in $\Sigma$ containing $\sigma$ as a face):
\begin{itemize}
    \item The affine open subset $U_\sigma$ is the union of orbits
    $$U_\sigma = \bigcup_{\tau \leq \sigma} \, O_\tau.$$
    \item $\tau \leq \sigma$ if and only if $O_\sigma$ is in the closure of $O_\tau$, and
    $$\overline{O_\tau} = \bigcup_{\tau \leq \sigma} \, O_\sigma$$
\end{itemize}
\end{thm}

Besides the $T$-torsor in $X_\Sigma$ corresponding to the smallest cone $\{0\} \in \Sigma$ which we call the \textit{main orbit}, all other orbits are themselves toric varieties for a quotient torus of $T$.

\begin{thm}[Proposition 3.2.7 \cite{CLS}]\label{thm: toricOrbits} With notation as in the previous Theorem, let $\tau \in \Sigma$ be a cone, and let $K_\tau = T/S_\tau$ be the quotient of $T$ acting on the orbit $O_\tau \subset X_\Sigma$ (so that $O_\tau$ is a $K_\tau$-torsor). The orbit closure $\overline{O}_\tau$ is isomorphic to the $K_\tau$-toric variety associated to the fan
    $$\Sigma_\tau := \big\{\sigma_{K_\tau} \in X_*(K_\tau)_\RR \, : \, \tau \leq \sigma \subset \Sigma\big\}$$
    where $(\, \cdot \, )_{K_\tau}: X_*(T)_\R \to X_*(K_\tau)_\R$ is the natural quotient map.
\end{thm}
For the rest of our discussion, a toric variety will always mean a \textit{normal and affine toric variety}, unless otherwise mentioned.

\subsection{Toric duality}
In this section, we define a natural duality structure 
$$\bigg\{\text{ affine } T\text{-toric varieties}\bigg\} \longleftrightarrow \bigg\{\text{ affine }\check{T}\text{-toric varieties}\bigg\}$$
which underlies the Langlands duality phenomenon that we are interested in. 

Let $\sigma \subset X_*(T)_\RR \cong X^*(\check{T})_\RR$ be a full-dimensional, rational, and strongly convex polyhedral cone. We consider the affine toric variety $X$ with torus $T$
$$X = \spec \ZZ\big[\check{\sigma} \cap X^*(T)\big]$$ 
denoted by $U_\sigma$ previously in \eqref{eq: affineToricDefn}. Let $\check{\sigma} \subset X^*(T)_\RR \cong X_*(\check{T})_\RR$; note that $\check{\sigma}$ is itself full-dimensional, rational, and strongly convex, so one may consider the affine toric variety $\check{X}$ with torus $\check{T}$
$$\check{X} = \spec \ZZ\big[\sigma \cap X^*(\check{T})\big]$$
Note that by definition, the monoids $\check{\sigma} \cap X^*(T)$ and $\sigma \cap X^*(\check{T})$ are both saturated, so the corresponding affine varieties $X$ and $\check{X}$ are both normal. 

\begin{defn}\label{defn: toricDual}
    Let $\sigma \subset X_*(T)_\RR$ be a rational strongly convex polyhedral cone, and let $\check{\sigma} \subset X_*(\check{T})_\R$ be its dual cone. Then we say that the $T$-toric variety $X = \spec \ZZ\big[\check{\sigma} \cap X^*(T)\big]$ and $\check{X} = \spec \ZZ\big[\sigma \cap X^*(\check{T})\big]$ are \textit{toric dual varieties}.
\end{defn}

Applying Theorem \ref{thm: orbitConeCorresp} to a pair of toric dual varieties, we see that the inclusion-reversing bijection on faces
$$\big\{\text{Faces of } \sigma\big\} \longleftrightarrow \big\{\text{Faces of } \check{\sigma}\big\}$$
$$\tau \longmapsto \tau^* := \tau^\perp \cap \check{\sigma}$$
induces an inclusion-reversing bijection on orbits (and more importantly for us, orbit closures) 
\begin{equation}\label{eq: orbitDuality}
    \big\{T\text{-orbits on }U_\sigma\big\} \longleftrightarrow \big\{\check{T}\text{-orbits on }U_{\check{\sigma}}\big\}
\end{equation}
$$O_\tau \longleftrightarrow O_{\tau^*} \, , \, \overline{O}_\tau \longleftrightarrow \overline{O}_{\tau^*}$$
We will write the above bijection on orbits and orbit closures as
$$O \longleftrightarrow O^* \, \text{ and } \, \overline{O} \longleftrightarrow \overline{O}^*$$
when explicit mention of the cone $\tau$ is cumbersome. 

Tracing through the definitions of this bijection, we may verify the following useful fact. In the notation of Theorem \ref{thm: orbitConeCorresp}, we have
$$O_\tau \cong T/S_\tau \text{ and } \overline{O}_\tau$$
are $T/S_\tau = K_\tau$-toric subvarieties. Considering the short exact sequence
$$1 \longrightarrow S_\tau \longrightarrow T \longrightarrow T/S_\tau =: K_\tau \longrightarrow 1$$
and dualizing to obtain 
$$1 \longrightarrow \check{K}_\tau \longrightarrow \check{T} \longrightarrow \check{S}_\tau \longrightarrow 1$$
we see that $\check{K}_\tau$ has cocharacter lattice $X_*(\check{K}_\tau) = \tau^\perp \cap X_*(\check{T})$; indeed, $\tau^\perp$ is exactly the $\RR$-span of $\check{T}$-cocharacters that pair trivially with $T$-cocharacters that lie in $\tau$. Taking the intersection of $\tau^\perp$ with the cone $\check{\sigma}$ gives a convex cone 
$$\tau^* = X_*(\check{K}_\tau)_\RR \cap \check{\sigma} \subset X_*(\check{K}_\tau)_\RR$$
which corresponds to the orbit(-closure) $O_{\tau^*}$ (resp. $\overline{O}_{\tau^*}$) of $\check{X}$ with the structure of a $\check{T}/\check{K}_\tau = \check{S}_\tau$-toric variety. In other words, the bijection \eqref{eq: orbitDuality} can be endowed with more structure, as a bijection of toric subvarieties for certain sub and quotient tori:
$$K_\tau\text{-toric variety } \overline{O}_\tau \subset X \longleftrightarrow \check{S}_\tau\text{-toric variety } \overline{O}_{\tau^*} \subset \check{X}$$

\section{Automorphic and spectral periods}

In this section, we review the definitions of automorphic periods, spectral periods, and weak numerical duality following the conventions of \cite{CV}. We shall formulate these notions at the natural generality of reductive groups, although for the purposes of the current discussion it suffices to consider split tori. We refer the reader to Appendix A of \textit{loc. cit} for the comparison with conventions of \cite{BZSV}.

\subsection{Automorphic periods}

Let $G$ be a reductive group and $X$ a $G$-variety, both of which defined over $\F_q$. We choose an eigenvolume form on $X$ whose eigenvalue is given by a character $\eta: G \to \Gm$. Consider the following unitarily-normalized action of the adelic points of $G \times \Ggr$ on the space of adelic Schwartz functions $\mathcal{S}(X(\mathbb{A}))$:
for $\underg = (g, \lambda) \in G \times \Ggr(\mathbb{A})$ we define:
    $$\underg \star \Phi(x) = |\eta(\underg)|^{1/2} \, \Phi(x \underg)$$
The normalized theta series of $X$ on $G(\mathbb{A})$ is defined by
    multiplying the Poincar{\'e} series $\sum_{x \in X(F)} \, (g,\partial^{1/2}) \star \Phi(x)$
    by a unitary normalization factor:

\begin{align} \label{thetaXdef}
        \theta_X (g) &=  
        \Delta^{\frac{\dim X- \dim G}{4}} \sum_{x \in X(F)} \, (g, \partial^{1/2}) \star \Phi^0(x),  
\end{align} 
where $\Phi^0$ is the characteristic function of integral points. {\em In what follows, if we write simply $\Phi$ without other definition,
we always mean the characteristic function of integral points.} 

Using this theta series we can define the \textit{(normalized) automorphic $X$-period}\footnote{This period depends on the choice of eigenvolume form on $X$, but we will customarily suppress this dependence in our notation.} of an unramified automorphic form $f$ on $G$ by 
   \begin{equation} \label{PXdef} P_X(f) := \int_{[G]} \, \theta_X(g)f(g) \, dg .
   \end{equation}  

As is well-known, the smaller-dimensional $G$-orbits in $X$ often contribute divergent quantities to the automorphic period $P_X$ above, and thus require careful regularization to define. We will postpone this delicate issue for now by defining $\mathring{X} \subset X$ to be the union of maximal dimensional $G$-orbits\footnote{For toric varieties with \textit{generically connected stabilizers}, this is just the unique dense $T$-orbit. We shall examine the removal of the connectedness assumption on generic stabilizers in Section \ref{appendix: DM}.}, and define the regularized theta series of $X$ by
$$\mathring{\theta}_X :=  \Delta^{\frac{\dim X- \dim G}{4}} \sum_{x \in \mathring{X}(F)} \, (g, \partial^{1/2}) \star \Phi(x)$$
and the \textit{regularized automorphic $X$-period} of an unramified automorphic form $f$ on $G$ by
 \begin{equation} \label{PXregdef} \mathring{P}_X(f) := \int_{[G]} \, \mathring{\theta}_X(g)f(g) \, dg .
\end{equation}  

\begin{rem}[Convergence of the regularized automorphic period]
    Suppose $X$ is a conical $G$-variety with respect to the $\Ggr$-action. By equivariantly embedding $X$ into a linear $G$-representation and removing the origin of the vector space (which we are allowed to do, since the cone point in $X$ does not contribute to the regularized period), one can dominante $|\mathring{P}_X(f)|$ by a product of Tate's integrals $\int_{[\Gm]} \, |t|^{s-1} \sum_{x \in F^\times} \, \Phi(xt)\,  dt$, ensuring the convergence of the former. See Lemma 2.3 of \cite{CV} for a detailed argument. 
\end{rem}

\subsection{Spectral periods and nonabelian $L$-functions}

Suppose we have an unramified $L$-parameter $\varphi: \Gamma \to \check{G}(\kk)$. Then $\Gamma$ acts on $\check{X}$ through $\varphi$, and we will assume that the fixed point set is discrete, otherwise
we regard the definition as invalid. We fix, as on the automorphic side, an eigenvolume form on $\check{X}$ whose eigencharacter is denoted by $\check{\eta}: \check{G} \to \Gm$. Note that $\check{\eta}$ may also be understood as a (central) cocharacter of $G$. 

Let $x_0 \in \check{X}(\kk)$ be an arbitrary fixed point of the $\Gamma$-action.
We then define the {\em local $L$-function attached to $(\check{X}, x_0)$ at a place $v$} to be a graded trace
\begin{equation}\label{eq: localLfunction}
    L_v(\check{X}, x_0, \varphi, s) := \mathrm{gtr}\big(\mathrm{Fr}_v \times q_v^{-s} \, | \, \widehat{\mathcal{O}}_{\check{X},x_0}\big) \in \kk[[q_v^{-s}]]
\end{equation}
where $\widehat{\mathcal{O}}_{\check{X},x_0}$ is the completed local ring of $\check{X}$ at $x_0$. 

We define the {\em normalized spectral period}
by taking an Euler product, and then summing over fixed points: 

\begin{equation} \label{LXdef} L_{\check{X}}(\varphi) := \mathfrak{z} \, 
\Delta^{\frac{\varepsilon - \dim \check{X}}{4}}  \sum_{x \in \mathrm{Fix}(\varphi,\check{X})}  
\prod_v L_v(\check{X}, x, \varphi, \frac{1}{2})\end{equation}
where $\varepsilon$ is the $\Ggr$-weight on the eigenmeasure, and $\mathfrak{z}$ is the scalar by which the (central element) $\check{\eta}(\partial^{-1/2}) \in G(\mathbb{A})$ acts on the automorphic representation of $G$ parametrized by $\varphi$.\footnote{In the toric case which is our main interest, for an unramified character $\chi: [T] \to \kk^\times$ the number $\mathfrak{z}$ is simply evaluated as $(\chi \circ \check{\eta})(\partial^{-1/2})$.}

\begin{rem}[Convergence of the spectral period] 
    Suppose each fixed point $x \in \mathrm{Fix}(\varphi, \check{X})$ is conical, i.e., the grading on $\widehat{\mathcal{O}}_{\check{X}, x}$ satisfies the conditions of Definition \ref{defn: conical}. Then the summand indexed by $x$ in the spectral period can be dominated by an $L$-function of Langlands type (see the comment following Lemma 3.1 of \cite{CV}), ensuring the convergence of the former. Such conditions are always satisfied in the cases of interest here.
\end{rem}

\subsection{Weak numerical duality and discrepancy}\label{subsect: weakDuality}

Let $(G_1, G_2)$ be a pair of split Langlands dual reductive groups. Recall that in \S 4 of \cite{CV} we defined the notion of \textit{weak numerical duality} between $G_i$-varieties $X_i$: informally speaking, $X_1$ and $X_2$ are weakly dual if we have the following equalities of periods evaluated on cusp forms
$$\mbox{automorphic $X_1$-period} \sim \mbox{spectral $X_2$-period} $$
$$\mbox{automorphic $X_2$-period} \sim \mbox{spectral $X_1$-period} $$
where $\sim$ means agreement up to a quarter-integer power of the discriminant $\Delta$. Formally, we make the following definition, suitably modified and made more explicit in the case of toric varieties we have at hand:

\begin{defn}\label{defn: weakDuality}
    Let $T_1, T_2$ be split Langlands dual tori. For $i = 1,2$, let $X_i$ be an affine $T_i \times \Ggr$-toric variety defined over $\ZZ$, each equipped with $T_i$-eigenvolume forms. We say that $(T_1, X_1)$ and $(T_2,X_2)$ are numerically weakly dual if for any finite field $\F$ of size $q$ and any curve $\Sigma$ over $\F$ with discriminant $\Delta$, we have equalities 
    \begin{equation}\label{eqn: toricWeakDuality}
        \mathring{P}_{X_1}(\chi_1) = \Delta^{a_{12}/4}L_{X_2}(\varphi_1) \, \text{ and } \, \mathring{P}_{X_2}(\chi_2) = \Delta^{a_{21}/4}L_{X_1}(\varphi_2)
    \end{equation}
    where
    \begin{itemize}
        \item $\chi_i$ is any everywhere unramified unitary Hecke character on $[T_i]$ with Langlands parameter $\varphi_i$,
        \item and the equalities are understood to hold whenever the fixed locus $\varphi_i$ on the $X_i$ are discrete.\footnote{In fact, if the fixed locus of $\varphi_i$ is discrete, then it will have a unique fixed point, which is the $T_i$-orbit corresponding to the full-dimensional cone.}
        \item $(a_{12}, a_{21}) \in \ZZ^2$ is a pair of integers, which we term the \textit{discrepancy} of the pair $(X_1,X_2)$. If $a:= a_{12} = a_{21}$ (which is the case in all our current applications), we simply say that the discrepancy is the integer $a$. 
    \end{itemize}
\end{defn}

\section{Weak numerical duality of affine toric varieties}

In this section we prove one of the main results of this discussion, that \textit{toric dual varieties are weakly numerically dual}. 

\begin{thm}\label{thm: main1}
    Let $T$ and $\check{T}$ be Langlands dual split tori of rank $r$, and let $X, \check{X}$ be toric dual varieties with torus $T$ and $\check{T}$ respectively, equipped with dual gradings in the sense of Definition \ref{defn: dualGrading}. Then $(T,X)$ and $(\check{T},\check{X})$ are weakly numercially dual with discrepancy $r - \varepsilon$ where $\varepsilon$ is the weight of the $\Ggr$-grading. 
\end{thm}

\subsection{Graded toric duality}

Note that the definition of both the automorphic and spectral periods depends on the action of an extra multiplicative group which we call $\Ggr$, and the choice of an eigenvolume form on the variety. We will spell out how we manage these parameters in the toric setting.

Let $T_1, T_2$ be split Langlands dual tori, and suppose we are given affine toric varieties $X_1, X_2$ with $T_1, T_2$-action, respectively. Suppose we fix a cocharacter
$$\check{\rho}: \Ggr \longrightarrow T_1$$
in order to view $X_1$ as a $T_1 \times \Ggr$-variety, and we choose a $T_1$-eigenvolume form on $X_1$ with $T_1$-eigencharacter
$$\eta: T_1 \longrightarrow \Gm$$
These choices induce the equivalent data of a cocharacter into the dual torus
$$\eta: \Ggr \longrightarrow T_2$$
with which we may equip $X_2$ with the structure of a $T_2 \times \Ggr$-variety. We also have a character 
$$\check{\rho}: T_2 \longrightarrow \Gm$$

\begin{lem}\label{lem: gradingAndVolumeForm}
    Suppose $X_1$ and $X_2$ are a pair of toric dual varieties in the sense of Definition \ref{defn: toricDual}. Let $x_1 \in X_1$ be the unique $T_1$-fixed point in $X_1$. Then for a cocharacter
    $$\check{\rho}: \Ggr \longrightarrow T_1$$
    the following conditions are equivalent:
    \begin{enumerate}
        \item The cocharacter $\check{\rho}$ defines a convergent local $L$-function at $(X_1, x_1)$ (defined by the expression \eqref{eq: localLfunction}) for $\mathrm{Re}(s) \gg 0$.
        \item The cocharacter $\check{\rho}$ makes $X_1$ a conical variety in the sense of Definition \ref{defn: conical}.
        \item The character $\check{\rho}: T_2 \longrightarrow \Gm$ is the $T_2$-eigencharacter of a $T_2$-eigenvolume form on $X_2$ with no poles.
    \end{enumerate}
    
\end{lem}
\begin{proof}
    Suppose $X_i$ are defined by cones $\sigma_i \subset X_*(T_i)_\RR$. From expression \eqref{eq: localLfunction} we see that $\check{\rho}$ defines a convergent local $L$-function at $(X_1,x_1)$ if and only if each level set of $\check{\rho}$ on $X^*(T_1)_\R$ contains finitely many lattice points of $\check{\sigma}_1 \cap X^*(T_1)$ and takes nonnegative values on them. This is equivalent to saying that $\check{\rho}$ lies in the cone $\sigma_1$ (to ensure nonnegativity), and does not lie in any proper faces of $\sigma_1$ (for the finiteness of lattice points in each level set). This corresponds precisely to the conditions of Definition \ref{defn: conical}, hence the equivalence of (1) and (2). 
    
    On the other hand, any lattice point in $\sigma_1$, now viewed as a cone in $X^*(T_2)_\RR$, would define an eigencharacter for an eigenvolume form of $T_2$, and vice versa. Indeed, such a lattice point $\lambda$ corresponds to a monomial $T_2$-eigenfunction $f$ on $X_2$ with eigencharacter $\lambda$, hence $fd_{T_2, \mathrm{Haar}}$ gives an eigenvolume form on $X_2$ with eigenvalue $\lambda$, where $d_{T_2, \mathrm{Haar}}$ is the Haar measure on the dense orbit $T_2 \subset X_2$. To see that $fd_{T_2, \mathrm{Haar}}$ is indeed well-defined on $X_2$, note first that it is well-defined on the dense open $T_2$-orbit. Now for every point $x$ in the closed complement, there is a point $y$ in the dense open orbit and a primitive cocharacter $\mu: \Gm \to T_2$ such that
    $$\lim_{t \to 0} \, y \cdot \mu(t) = x$$
    To see that $fd_{T_2, \mathrm{Haar}}$ is regular at $x$ we pull-back via $\mu$ and verify that
    $$\mu^*fd_{\mathrm{T_2}, \mathrm{Haar}}(t) = t^{\langle \lambda, \mu\rangle-1} \, dt$$
    Since $\langle \lambda, \mu\rangle \geq 1$ for all cocharacters $\mu$, we see that $fd_{T_2, \mathrm{Haar}}$ is regular everywhere. 
\end{proof}

In view of Lemma \ref{lem: gradingAndVolumeForm}, we make the following definition.
\begin{defn}\label{defn: dualGrading}
    Let $X_i$ be $T_i$-toric varieties which are toric duals, for $i = 1,2$. We say that $(T_i, X_i)$ are equipped with \textit{dual gradings} if we further choose cocharacters $\check{\rho}_1: \Ggr \to T_1$ and $\check{\rho}_2: \Ggr \to T_2$, and we assume that $\check{\rho}_1,\check{\rho}_2$ satisfy the equivalent conditions of Lemma \ref{lem: gradingAndVolumeForm}. We understand that $X_i$ is equipped with the corresponding $T_i \times \Ggr$-actions and the $T_i$-eigenvolume forms. 
\end{defn}
\begin{rem}\label{rem: weightOfGrading}
    Note that the $\Ggr$-weight on the eigenvolume form of both $X_1$ and $X_2$ (equipped with dual gradings $(\check{\rho}_1, \check{\rho}_2)$) can be computed as
    $$\varepsilon = \langle \check{\rho}_1, \check{\rho}_2\rangle \in \ZZ$$
    We will simply refer to this positive integer $\varepsilon$ as the \textit{$\Ggr$-weight of the dual grading}. Note that it is always possible to pick them so that $\varepsilon$ matches $r = \mathrm{dim}(X) = \mathrm{dim}(\check{X})$.
\end{rem}

We will assume for the rest of this section that toric varieties are \textit{graded}, i.e., that for a $T$-toric variety $X$ we have chosen 
\begin{itemize}
    \item a cocharacter $\check{\rho}: \Ggr \to T$ to view $X$ as a $T \times \Ggr$-variety, and
    \item a volume eigenform $\eta$, whose eigencharacter we also denote by $\eta$, associated to some lattice point in the dual of the cone defining $X$. 
\end{itemize}
In particular, we will view $\Ggr$ as a subgroup of $T$ without explicitly writing the cocharacter $\check{\rho}$.

\subsection{Weak numerical duality}
Note that the factor of $\Delta^{\frac{\mathrm{dim} \, X - \mathrm{dim} \, T}{4}}$ in the definition of $\theta_X$ in \eqref{PXdef} disappears, and the (regularized) automorphic period is by definition
\begin{equation}
    \mathring{P}_X(\chi) = |\eta(\partial^{1/2})|^{1/2} \, \int_{[T]} \, \chi(t) \, |\eta(t)|^{1/2}\sum_{x \in T(F)} \, \Phi(xt\partial^{1/2}) \, dt
\end{equation}
Making the change of variables $t \mapsto t\partial^{-1/2}$, we have
\begin{equation}
        \mathring{P}_X(\chi) = \chi(\partial^{-1/2}) \, \int_{[T]} \, \chi|\eta|^{1/2}(t)\sum_{x \in T(F)} \, \Phi(xt) \, dt
\end{equation}
Unfolding in the usual way we see that the regularized automorphic period Euler factorizes as follows: taking an arbitrary rational point $x_O \in O$,
\begin{align*}\label{eqn: eulerFactors}
\mathring{P}_X(\chi) &= \chi(\partial^{-1/2})\int_{T(\mathbb{A})} \, \chi|\eta|^{1/2}(t)\Phi(x_Ot) \, dt \\
&= \chi(\partial^{-1/2})\prod_v \, \int_{T(F_v)} \, \mathbf{1}_{X(\mathcal{O}_v)}(x_Ot)\chi|\eta|^{1/2}(t) \, dt\\
    &= \chi(\partial^{-1/2})\prod_v \, \sum_{\lambda \in X_*(T)} \, \mathbf{1}_{X(\mathcal{O}_v)}(x_O\varpi_v^\lambda)\chi_v(\varpi_v^\lambda) \, q_v^{-\langle \eta,\lambda\rangle /2}
\end{align*}
The Schwartz function $\mathbf{1}_{\mathcal{O}_v}$ exactly picks out those $\lambda$ such that
$$\langle \lambda, x\rangle \geq 0 \text{ for all } x \in \check{\sigma} \cap X^*(T)$$
Indeed, the point $\varpi^\lambda_v \in T(F_v) \subset X(F_v)$ lies in $X(\mathcal{O}_v)$ if all integral functions take integral values on $\varpi^\lambda_v$. But this is exactly the case if all the characters in $\check{\sigma} \cap X^*(T)$ pair positively with $\lambda$. So each Euler factor here can be written as
\begin{equation}\label{eq: mainOrbitFinal}
    \sum_{\lambda \in X_*(T): \langle \lambda, \check{\sigma}\rangle \geq 0} \, \chi(\varpi^\lambda_v)q^{-\langle \eta, \lambda\rangle/2}_v = \mathrm{gtr}\big(\mathrm{Fr}_v \times q_v^{-1/2} | \, \kk[\sigma \cap X^*(\check{T})]\big)
\end{equation}
since
$$\lambda \in X_*(T) \text{ with } \langle \lambda, \check{\sigma} \rangle \geq 0 \iff \lambda \in \sigma \subset X_*(T)_\RR \cong X^*(\check{T})_\RR$$
The grading on $\kk[\sigma \cap X^*(\check{T})]$ with respect to which we take Frobenius trace on the right hand side of \eqref{eq: mainOrbitFinal} is given by $|x| = \langle \eta,x\rangle$ for all $x \in X^*(\check{T})$. We have thus finished the demonstration of Theorem \ref{thm: main1}.

\begin{rem}
    We see from Theorem \ref{thm: main1} that even for singular toric varieties, one may choose dual gradings so that we have weak numerical duality \textit{without discrepancy}. Contrast this with the singular examples of \cite{CV} and the Eisenstein example (see Appendix B.4.2 of \textit{op. cit}).
\end{rem}

\begin{rem}[Choice of Schwartz function] \label{rem: BNS}
    When $X$ is singular, there are (at least) two natural choices for the Schwartz function in defining the automorphic period using Equation \eqref{PXdef}: one may take the ``IC-function of generically smooth arcs" 
    $$\Phi^{\mathrm{IC}} = \otimes_v \, \mathrm{IC}_{\mathcal{L}^\circ_v X}$$
    as defined in the work of Bouthier--Ngô--Sakellaridis (see Section 1 of \cite{Bouthier-Ngo-Sakellaridis}), or one may take the ``indicator function of arc space" 
    $$\Phi^0 = \otimes_v \, \mathbf{1}_{X(\mathcal{O}_v)}$$ 
    as we have done here and as is done in \cite{CV}. Loosely speaking, the latter counts ``sections of the associated $X$-bundle of a principal $G$-bundle", while the former is the function-theoretic analogue of the intersection cohomology complex of the generically smooth arc space of $X$.
    
    Both options have their advantages. In this article and in \textit{loc. cit}, we argue that the $\Phi^0$ may be better suited for studying weak numerical duality, especially if one admits the notion of nonabelian $L$-functions. For more discussion on this topic we refer the reader to \S 1.4 and Remark 2.2 of \textit{loc. cit}.
\end{rem}

\section{Regularization of small orbits}\label{sect: regularize}

Let $\{O\}$ be the set of $T(F)$-orbits on $X(F)$, $x_O$ be an arbitrary base point in the orbit $O$, and $S_O$ be the stabilizer of $x_O$ in $T$. Formally, the \textit{unregularized} automorphic period \eqref{PXdef} can be unfolded as
\begin{equation}\label{eq: orbitDecomp}
    P_X(\chi) = \chi(\partial^{-1/2}) \,\sum_O \, \int_{S_O(\mathbb{A}) \backslash T(\mathbb{A})} \, \Phi(x_O t) \int_{[S_O]} \, \chi|\eta|^{1/2}(st) \, ds \, dt
\end{equation}

Our main goal in this Section is Theorem \ref{thm: main2}, in which we upgrade Theorem \ref{thm: main1} to a more precise duality which captures information about all orbits in \eqref{eq: orbitDecomp} via regularization procedures on both the automorphic and spectral sides.

\subsection{Automorphic period}\label{subsect: otherOrbits}

Note that when $\chi|\eta|^{1/2}|_{[S_O]}$ is nontrivial, the inner integral over $[S_O]$ in \eqref{eq: orbitDecomp} formally vanishes. Thus we make the following 
\begin{defn}
    Let $\chi: [T] \to \kk^\times$ be a character, and $X$ a graded $T$-toric variety, equipped with an eigenform with eigenvalue $\eta \in X^*(T)$. Let $\{O\}$ be the set of $T$-orbits on $X$. For each orbit $O$, we write $S_O \subset T$ for the stabilizer of a generic point in the orbit $O$. We say $\chi$ is \textit{$(X,\eta)$-cuspidal} if for all proper stabilizer subtori $S_O \subsetneq T$ we have $\chi|\eta|^{1/2}|_{[S_O]} \neq \mathbf{1}$.
\end{defn}
As the eigenform $\eta$ is always understood implicitly, we will often say \textit{$X$-cuspidal} instead. 
\begin{rem}
    Of course, cuspidality is not a useful notion for automorphic forms on $T$, but we make this definition in allusion to the fact that cuspidality is also the condition of the vanishing of certain periods. More precisely, in the general case when one considers a reductive group $G$ and the automorphic period attached to a $G$-space $X$, one sees a similar orbit decomposition as in Equation \eqref{eq: orbitDecomp}. Usually one discards the orbits that are not dense by appealing to cuspidality. 
\end{rem}

Let $O$ be an orbit which is not dense in $X$; we formally calculate its contribution to the unregularized automorphic period. Note that $O$ is itself a torsor for a quotient torus of $T$; namely, if $S \subset T$ is the stabilizer of a point in $O$, then $O$ is a $K := T/S$-torsor. It contributes the following term to the automorphic period $P_X(\chi)$ as in \eqref{eq: orbitDecomp}:
$$\chi(\partial^{-1/2}) \int_{[T]} \chi|\eta|^{1/2}(t) \sum_{x \in O(F)} \, \Phi(xt) \, dt = \chi(\partial^{-1/2})\int_{K(\mathbb{A})} \, \Phi(x_Ot) \int_{[S]} \, \chi|\eta|^{1/2}(st) \, ds \, dt$$
As mentioned above, if $\chi|\eta|^{1/2}$ is nontrivial on $[S]$ then the inner integral vanishes identically; therefore, we assume that $\chi|\eta|^{1/2}|_{[S]} = \mathbf{1}$ for this calculation and note that we obtain an Euler-factorizable expression
\begin{equation} \label{eqn: formalZeta1}
    \mathrm{vol}([S])\, \int_{K(\mathbb{A})}\, \Phi(x_Ot)\chi|\eta|^{1/2}(t) \, dt =  \mathrm{vol}([S]) \prod_v \,\int_{K_v} \mathbf{1}_{\mathcal{O}_v}(x_Ot)\chi|\eta|^{1/2}(t) \, dt
\end{equation}
Of course, $\mathrm{vol}([S])$ gives a divergent factor $\zeta(1)^{\mathrm{dim}(S)}$, but we carry it around formally for now to see that it matches an equivalent divergent factor on the spectral side. 

Note that each Euler factor now resembles a main term calculation for the $K$-toric variety $\overline{O}$; in particular, according to Theorem \ref{thm: toricOrbits}, if $\sigma$ is the cone corresponding to the $T$-orbit (closure) $\overline{O}$, then as a $K$-toric variety it is associated to the cone 
$$\sigma_K := \text{ image of }\sigma \text{ under the projection } X_*(T)_\RR \to X_*(K)_\RR$$
The same argument as in the main term calculation computes each Euler factor as
$$\mathrm{gtr}\big(\mathrm{Fr}_v \times q_v^{-1/2} \, | \, \kk[\sigma_K \cap X^*(\check{K})] \big)$$
where $\check{K} = \mathrm{Ker}(\check{T} \to \check{S})$ is the dual torus to $K$.

\subsection{Spectral period}

Let us now consider the spectral period of $\check{X}$. On the automorphic side we suggested a notion of \textit{cuspidality relative to $X$}; on the spectral side we introduce a notion of \textit{irreducibility relative to $\check{X}$} for $L$-parameters. 
\begin{defn} 
    Let $\varphi: \Gamma_F \to \check{T}$ be an $L$-parameter, and $\check{X}$ a graded $\check{T}$-toric variety, with grading $\rho: \Ggr \to \check{T}$. We say $\varphi$ is \textit{$(\check{X}, \rho)$-generic} if the action of $\Gamma_F$ on $\check{X}$ via 
    $$\Gamma_F \overset{\varphi \times \varpi^{1/2}}{\longrightarrow} \check{T} \times \Ggr \overset{1 \times \rho}{\longrightarrow} \check{T} \acts \check{X}$$
    has a unique fixed point on $\check{X}$, where $\varpi^{1/2}$ is a choice of square root of the cyclotomic character.
\end{defn}
As the grading $\rho$ is always understood implicitly, we will often say \textit{$\check{X}$-generic} instead. 
\begin{rem} \label{rem: extendedParameter}
    The homomorphism $\varphi \times \varpi^{1/2}: \Gamma_F \to \check{G} \times \Ggr$, for $\varphi$ an $L$-parameter into a general reductive group $\check{G}$, is called an \textit{extended} $L$-parameter in the terminology of \cite{BZSV}. The idea is that the norm character $| \, \cdot \, |$ corresponds under our normlization of class field theory (applied to $\Ggr$) to the cyclotomic character $\varpi$.
\end{rem}

The following Lemma is immediate from definitions, and plays the role of the familiar heuristic that the $L$-parameter of a cuspidal automorphic form should have big image in the dual group. 
\begin{lem}\label{lem: cuspidalImpliesUniqueFixPoint}
    Let $(T, X)$ and $(\check{T}, \check{X})$ be a graded toric dual pair. Then an unramified character of $[T]$ is $X$-cuspidal if and only if its $L$-parameter is $\check{X}$-generic. 
\end{lem}
\begin{proof}
    We write $\eta$ for the eigenform on $X$, and $\rho: \Ggr \to \check{T}$ be the data given by the graded toric pair. For an orbit $O$ on $X$, we write $S_O$ as its stabilizer and $K_O = T/S_O$ as its quotient torus. We write $\pi_O: \check{T} \to \check{S}_O$ for the quotient of dual tori. For a character $\chi$ of $[T]$ we will write $\varphi(\chi): \Gamma_F \to \check{T}$ for its $L$-parameter. Then we see that
    \begin{align*}
        \chi \text{ is } (X,\eta)\text{-cuspidal} &\iff \text{ for all } O \text{ with } S_O \subsetneq T, \chi|\eta|^{1/2}|_{[S_O]} \neq \mathbf{1}_{[S_O]}\\
        &\iff \varphi(\chi \times |\eta|^{1/2}) = \varphi(\chi) \times \rho \circ \varpi^{1/2} \\
        & \quad \quad \quad \text{ does not factor through } \check{K}_O = \mathrm{Ker}(\pi_O)
    \end{align*}

    By the orbit-cone correspondence (Theorem \ref{thm: orbitConeCorresp}), we see that $\varphi(\chi \times |\eta|^{1/2})$ being not contained in any of the proper subtori of $\check{T}$ associated to the proper faces of $\check{\sigma}$ is equivalent to it not fixing pointwise any other $\check{T}$-orbits besides the one associated to $\check{\sigma}$ itself, which is the unique $\check{T}$-fixed point. 
\end{proof}

\subsection{Regularization and period duality}
In the case when $\chi$ is not $X$-cuspidal, we will give a regularization of automorphic and spectral periods which allows us to ``formally cancel $\zeta(1)$'s". 

As is usual in the regularization of automorphic periods, the goal is to interpret the period integral as an intertwiner of representations. Let $O \subset X$ be a $T$-orbit which is nondense; we pick an arbitrary base point $x_O \in O$ to identify $T/S \cong O$ where $S = \mathrm{Stab}_T(x_O)$, and we write $K = T/S$ as the quotient torus. Formally speaking, we have some ``Schwartz space" denoted by $\mathcal{S}(O_\mathbb{A})$, which at least has to contain the restriction of the function $\mathbf{1}_{X(\mathcal{O})}$ to $O_\mathbb{A}$. Then the formal integral in \S \ref{subsect: otherOrbits} represents an intertwiner
$$\bigg[P_O(\chi):\Phi \mapsto \int_{[T]} \chi|\eta|^{1/2}(t)\sum_{x \in O(F)} \, \Phi(xt) \, dt\bigg] \in \mathrm{Hom}_{T(\mathbb{A})}(\chi|\eta|^{1/2} \otimes \mathcal{S}(O(\mathbb{A})), \kk)$$
In the event that $\chi|\eta|^{1/2}|_{S(\mathbb{A})} = \mathbf{1}_{S(\mathbb{A})}$, we can quotient by $S$ and rewrite this Hom-space as
$$P_O(\chi) \in \mathrm{Hom}_{K(\mathbb{A})}(\chi|\eta|^{1/2} \otimes \mathcal{S}(\mathcal{O}_\mathbb{A}), \kk)$$
This heuristic calculation justifies the following
\begin{defn}\label{def: regularizeAut}
   Let $(X,\eta)$ be a graded $T$-toric variety with eigenform $\eta$. Let $O \subset X$ be a $T$-orbit. Let $S \subset T$ be the stabilizer of a rational point in $O$, and let $K = T/S$. We define the \textit{regularized automorphic contribution of the orbit $O$} of a character $\chi:[T] \to \kk^\times$ as follows:
\begin{equation}\label{eqn: regAutPeriod}
    P_{O \subset X}^{\mathrm{reg}}(\chi) :=  \chi(\partial^{-1/2}) \cdot \begin{cases} \, \mathring{P}_O(\chi) = \int_{K(\mathbb{A})} \, \Phi(x_O t)\chi|\eta|^{1/2}(t) \, dt \, \\ \quad \quad \quad \quad  \text{ if } \chi|\eta|^{1/2}|_{[S]} = \mathbf{1} \text{ and } \chi|\eta|^{1/2}|_{[K]} \text{ is } \overline{O}\text{-cuspidal}\\
     \, 0 \, \text{ otherwise}
    \end{cases}
\end{equation}
where $\Phi = \otimes_v \, \mathbf{1}_{X(\mathcal{O}_v)}$.
\end{defn}
In particular, the regularized automorphic contribution of a non-dense orbit $O$ of a character $\chi$ is nonzero only if $\chi$ fails to be $X$-cuspidal.

On the spectral side, we also examine the contribution of orbits one by one. In the case when the $L$-parameter is generic, only the smallest orbit, i.e., the $\check{T}$-fixed point will contribute. If the $L$-parameter factors through a subtorus making it nongeneric, then it must fix pointwise some non-minimal orbit. 

\begin{defn}\label{def: regularizeSpec}
    Let $(\check{X}, \rho)$ be a graded $\check{T}$-toric variety with grading $\rho: \Ggr \to \check{T}$. Let $\check{\sigma} \in X_*(\check{T})_\RR$ be the cone defining $\check{X}$ as a toric variety, whose dual cone is $\sigma \subset X^*(\check{T})_\R$. Suppose $\varphi_0: \Gamma_F \to \check{T}$ is an $L$-parameter whose extended $L$-parameter (see Remark \ref{rem: extendedParameter}) we denote by $\varphi$, and let $O \subset \check{X}$ be a $\check{T}$-orbit with stabilizer $\check{K} \subset \check{T}$. Then we define the \textit{regularized spectral contribution of the orbit $O$} of $\varphi$, denoted $L_{O \subset \check{X}}^{\mathrm{reg}}(\varphi)$, as follows: 
    \begin{itemize}
        \item If $\varphi$ does not fix $O$ pointwise, then we set
    $$L_{O \subset \check{X}}^{\mathrm{reg}}(\varphi) = 0$$
        \item  Otherwise, suppose $\varphi$ fixes $O \subset \check{X}$ pointwise. If there is another $\check{T}$-orbit $O' \neq O$ for which $O \subset \overline{O}'$ and such that $\varphi$ fixes $O'$ pointwise, then we also set
    $$L_{O \subset \check{X}}^{\mathrm{reg}}(\varphi) = 0$$
        \item     Finally, suppose $\varphi$ fixes $O$ pointwise, and $O$ is not contained in any other $\varphi$-pointwise-fixed orbit closure. Then $\varphi$ is valued in $\check{K}$ and we set
    \begin{equation}\label{eqn: regSpecPeriod} L_{O \subset \check{X}}^{\mathrm{reg}}(\varphi) = \mathfrak{z} \,  \Delta^{\frac{\varepsilon - \mathrm{dim}(\check{X})}{4}}\prod_v \, \mathrm{gtr}\big(\varphi(\mathrm{Fr}_v) \times q_v^{-1/2} | \, \kk[\sigma_{\check{K}} \cap X^*(\check{K})]\big)\end{equation}
    where $\sigma_{\check{K}}$ is the image of $\check{\sigma}$ under the surjection $X^*(\check{T})_\R \twoheadrightarrow X^*(\check{K})_\R$, and $\mathfrak{z}$ is the scalar defined in \eqref{eqn: unramSpecPeriod}.
    \end{itemize}
   
\end{defn}
The second bulleted case is handled as such to avoid having to ``cancel extra $\zeta(1)$'s"; it is analogous to requiring that a character be $\overline{O}$-cuspidal for the orbit $O$ to contribute to the automorphic period in Definition \ref{def: regularizeAut}. 

\begin{rem}
    Note that our definition of regularized automorphic contribution of an orbit $O$ recovers the main $T$-orbit in $X$ of Equation \eqref{eq: orbitDecomp}, and that our definition of regularized spectral contribution of an orbit $O$ recovers \eqref{LXdef} when $O$ is the unique $\check{T}$-fixed point in $\check{X}$.

    In fact, the argument of Lemma \ref{lem: cuspidalImpliesUniqueFixPoint} can be refined using the orbit-cone correspondence: an orbit $O \subset X$ contributes to the automorphic period of $\chi: [T] \to \CC^\times$ if and only if its extended $L$-parameter factors through $\check{K}_O$, in which case it fixes pointwise the orbit $O^* \subset \check{X}$ (see \eqref{eq: orbitDuality} again for the definition of the correspondence of orbits in toric dual varieties). 
\end{rem}

\begin{rem}
    Morally speaking, and literally so in the case when the normal sheaf $\mathcal{N}_{\overline{O}/\check{X}}$ of $\overline{O}$ is a vector bundle (or perfect complex) over $\overline{O}$, the regularized spectral contribution of $O$ of an $L$-parameter $\varphi$ fixing $O$ pointwise is the $L$-function of $\varphi$ on $\mathcal{N}_{\overline{O}/X}|_x$ where $x \in \overline{O}$ is a generic point in the orbit $O$. 
    
   To explain in slightly more detail, we let $i: \overline{O} \hookrightarrow \check{X}$ be the closed embedding of the orbit closure, and consider the distinguished triangle of tangent complexes at the point $x$
    $$\mathbf{T}i|_x \longrightarrow  \mathbf{T}_{\overline{O},x} \longrightarrow i^*\mathbf{T}_{X,x} \overset{+1}{\longrightarrow}$$
    By a result of Quillen (see Appendix A.2 of \cite{CV} for a complete discussion), the alternating product of $L$-functions attached to the cohomologies of $\mathbf{T}i|_x$ is given by
    \begin{equation}\label{eq: regularizedSpectralPeriodInterp}
        \prod_q \, \mathrm{det}(1-t \,\mathrm{Fr}_v \, | \, \mathbf{T}_qi|_x)^{(-1)^{q+1}} = \frac{\mathrm{gtr}\big(t \,\mathrm{Fr}_v \, | \, \widehat{\mathcal{O}}_{\check{X},x}\big)}{\mathrm{gtr} \big(t \, \mathrm{Fr}_v \,| \, \widehat{\mathcal{O}}_{\overline{O},x}\big)}
    \end{equation}
    where the power of $t$ is keeping track of the grading on both sides. When the $L$-parameter factors through $\check{K} \subset \check{T}$, the denominator is just the Hilbert series of graded dimensions of the ring $\widehat{\mathcal{O}}_{\overline{O},x}$ since $\mathrm{Fr}_v$ will act trivially, which is equal to the local zeta function at $1$ to the $\mathrm{dim}(\check{S})$th power. In this case, we observe that the right hand side of \eqref{eq: regularizedSpectralPeriodInterp} becomes
    $$\frac{\mathrm{gtr}\big(t \, \mathrm{Fr}_v \, | \, \CC[\sigma \cap X^*(\check{T})]\big)}{\zeta_v(1)^{\mathrm{dim}(\check{S})}} = \mathrm{gtr}\big(t \, \mathrm{Fr}_v \, | \, \CC[\sigma_{\check{K}} \cap X^*(\check{K})]\big)$$
    where $\sigma_{\check{K}}$ is the image of $\sigma$ under the projection $X^*(\check{T})_\RR \to X^*(\check{K})_\RR$. This recovers our definition of regularized spectral contribution, and note that if we take a product over all places $v$, we have divided out formally by a factor of $\zeta(1)^{\mathrm{dim}(\check{S})}$ which formally matches the $\zeta(1)^{\mathrm{dim}(S)}$ in Equation \eqref{eqn: formalZeta1}.
\end{rem}

\begin{defn}
    Let $(T, X)$ and $(\check{T}, \check{X})$ be a pair of affine toric varieties. We say that $(T,X)$ and $(\check{T}, \check{X})$ have \textit{Langlands dual periods} with discrepancy $(a, \check{a}) \in \ZZ \times \ZZ$ if for every $T$-orbit $O$ and $\check{T}$-orbit $O^*$ that correspond to each other under \eqref{eq: orbitDuality}, we have
    $$P_{O \subset X}^{\mathrm{reg}}(\chi) = \Delta^{a/4} \, L_{O^* \subset \check{X}}^{\mathrm{reg}}(\varphi_\chi)$$
    for all characters $\chi: [T] \to \CC^\times$ with $L$-parameter $\varphi_\chi$, and
    $$P_{O^* \subset \check{X}}^{\mathrm{reg}}(\mu) = \Delta^{\check{a}/4}L_{O \subset X}^{\mathrm{reg}}(\varphi_\mu)$$
    for all characters $\mu: [\check{T}] \to \CC^\times$ with $L$-parameter $\varphi_\mu$.
\end{defn}

Combining the formal calculation of \eqref{eq: orbitDecomp} and applying Theorem \ref{thm: main1} to each orbit contribution, we arrive at the following refinement of weak numerical duality:
\begin{thm}\label{thm: main2}
       Let $T$ and $\check{T}$ be Langlands dual split tori of rank $r$, and let $X, \check{X}$ be toric dual varieties with torus $T$ and $\check{T}$ respectively, equipped with dual gradings in the sense of Definition \ref{defn: dualGrading}. Then $(T,X)$ and $(\check{T},\check{X})$ have Langlands dual periods with discrepancy $r - \varepsilon$ where $\varepsilon$ is the weight of the dual grading.
\end{thm}

\section{Towards weak numerical duality of Deligne--Mumford stacks} \label{appendix: DM}

In this section we suggest provisional definitions of weak numerical duality between toric varieties with finite stabilizers and toric Deligne--Mumford stacks which are global quotients of a toric variety by a finite group. These examples are derived from a weakly numerically dual pair of toric varieties, so in some sense we do not compute any new periods, but we will see arithmetic phenomena that serve as guidelines to treat more sophisticated examples. 

For motivation, consider a $T$-variety $X$ with a dense $T$-orbit $\mathring{X}$, which we place on the automorphic side. Note that the formal unfolding as in \eqref{eq: orbitDecomp} of the \textit{regularized} automorphic period continues to hold even if $T(F)$ has multiple orbits on $\mathring{X}(F)$; this occurs, for instance, when the generic $T$-stabilizer on $X$ is disconnected. Indeed, if $\{a\}$ is a set of orbit representatives of the action of $T(F)$ on $\mathring{X}(F)$, we may formally unfold the regularized automorphic period \eqref{PXregdef} as
\begin{equation}\label{eq: arithmeticOrbits}
    \mathring{P}_X(\chi) = \chi(\partial^{-1/2})\sum_a \, \int_{S_a(\mathbb{A})\backslash T(\mathbb{A})} \Phi(at) \int_{[S_a]}\, \chi|\eta|^{1/2}(st) \, ds \, dt
\end{equation}
where $S_a \in T$ is the stabilizer of $a$. We treat this equation formally as the data of an integral attached to each orbit of $T(F)$ on $\mathring{X}(F)$, as the sum over $a$ may well be an infinite one and may not converge. However, if one restricts to only \textit{unramified orbits} (see the discussion surrounding Equation \eqref{eqn: canonicalBijection}), then the sum is much smaller and our basic proposal is that
\begin{quote}
    The unramified orbit contributions to $\mathring{P}_X$ correspond to unramified twists of the $L$-function.
\end{quote}
This is to be made precise by Definition \ref{defn: weakDualityStack} and Theorem \ref{thm: DMTheorem}.

\subsection{Periods of Deligne--Mumford stacks}

Let $\mu \subset T$ be a finite subgroup scheme. Then we can consider the quotient torus
$$1 \longrightarrow \mu \longrightarrow T \overset{\pi}{\longrightarrow} T' := T/\mu \longrightarrow 1$$
The dual of this sequence reads
$$1 \longrightarrow \check{\mu} \longrightarrow \check{T}' \overset{\check{\pi}}{\longrightarrow} \check{T} \longrightarrow  1$$
where $\check{\mu} := \mathrm{Hom}(\mu, \Gm)$. 

Now suppose we have $(T', Y)$ is a graded $T'$-toric variety, which we consider as a $T$-space via $\pi$; when we consider $Y$ as a $T$-space we will write it as $X$ to avoid confusion. Note that the stabilizer in $T$ of a point in $\mathring{X}$ is the disconnected group $\mu$. Nonetheless, the $T'$-eigenform $\eta'$ and grading (given by a cocharacter $\check{\rho}': \Ggr \to T'$) induces a $T$-eigenform $\eta$ and one can choose a grading on $X$ compatible with the choice of $\check{\rho}'$ (the latter given by an arbitrary lift of $\check{\rho}'$ to a $\Q$-cocharacter $\check{\rho} \in X_*(T)_\Q$). Note that the choice of lift of $\check{\rho}'$ will not affect the formula defining the automorphic $X$-period and the spectral $X$-period since the $\Ggr$-action factors through $T'$, so in the end it depends only on $\check{\rho}'$.

We will call the collection of data
\begin{equation}
    (T,X, \eta, \check{\rho})
\end{equation}
with the choice of lift $\check{\rho}  \in X_*(T)_\QQ$ an \textit{induced grading} on the $T$-toric variety\footnote{Strictly speaking, since $T$ acts with finite stabilizers on $X$, it is not toric in the standard sense. But we abuse terminology here slightly to refer to $X$ as $T$-toric since it contains a dense $T$-orbit.} constructed out of $(T', Y)$.

We propose that the (weak numerical) dual of $(T,X)$ ought to be the Deligne--Mumford stack
\begin{equation} \label{eqn: toricDMStack}
    \check{X} := [\check{Y}/\check{\mu}] =  \check{Y} \times^{\check{T}'} \check{T}
\end{equation}
where $(\check{T}',\check{Y})$ is the graded toric dual of $(T', Y)$. Note that $\check{T}$ acts on $\check{X}$ via the formula
$$y \cdot t := y \tilde{t}$$
where $\tilde{t}$ is an arbitrary lift of $t \in \check{T}$ to $\check{T}'$. Since $\check{Y}$ has a $\check{T}'$-grading given by some cocharacter $\rho': \Ggr \to \check{T}'$, we also have an induced cocharacter $\rho = \check{\pi} \circ \rho': \Ggr \to \check{T}$. On the other hand, the $\check{T}'$-eigenform $\check{\eta}'$ on $\check{Y}$ does not canonically give rise to a $\check{T}$-eigenform. One has to choose a $\Q$-character $\check{\eta}: \check{T} \to \Gm$ in $X^*(\check{T})_\Q$ such that 
$$\check{\eta}' = \check{\eta} \circ \check{\pi}$$
\begin{defn}
    Let $(T', Y, \eta, \check{\rho}')$ be a toric $T'$-variety, whose graded toric dual is $(\check{T}', \check{Y}, \check{\rho}', \eta)$. From this we construct the graded $T$-toric variety $(T,X, \eta, \check{\rho})$ with a choice of $\check{\rho} \in X_*(T)_\QQ$ lifting the $T'$-grading on $Y$. Then $\eta$ gives rise to a $\check{T}$-valued cocharacter with which we can use as a grading on $\check{X}$, and $\check{\rho}$ gives rise to a rational character of $T$ compatible with the eigencharacter of $(\check{T}', \check{Y})$. We say that the data
    $$(T,X, \eta, \check{\rho}) \text{ and } (\check{T}, \check{X}, \check{\rho}, \eta)$$
    with the choice of $\check{\rho} \in X_*(T)_\QQ = X^*(\check{T})_\QQ$, is an \textit{induced dual grading} on $X$ and $\check{X}$.
\end{defn}

We are now ready to develop the notions of automorphic and spectral periods for $X$ and $\check{X}$. Note that the Definitions \eqref{PXdef} for the automorphic period and \eqref{LXdef} for the spectral period make sense for $(T,X)$ even though it has disconnected stabilizers. Thus we focus on defining analogous objects for the stacky $(\check{T}, \check{X})$.

\begin{defn}
    Let $\varphi: \Gamma_F \to \check{T}$ be an $L$-parameter. A fixed point of the $\varphi$-action of $\Gamma_F$ on $\check{X} = \check{Y} \times^{\check{T}'} \check{T}$ is a fixed point in the groupoid sense, i.e., (the groupoid of) pairs $(y, \widetilde{\varphi}) \in \check{Y}(\kk) \times \mathrm{Hom}(\Gamma_F, \check{T}')$ such that $\widetilde{\varphi}$ is a lift of $\varphi$ to $\check{T}'$ and $y \in \check{Y}$ is a fixed point of the $\widetilde{\varphi}$-action. 

    Two fixed points $(y, \widetilde{\varphi})$ and $(y', \widetilde{\varphi}')$ are isomorphic if there is an isomorphism of $\Gamma_F$-actions $\alpha: \widetilde{\varphi} \overset{\sim}{\to} \widetilde{\varphi}'$ bringing $y$ to $y'$. We say that a fixed point $(y, \widetilde{\varphi})$ is \textit{unramified} if $\widetilde{\varphi}$ is an unramified lift of $\varphi$.
\end{defn}
\begin{rem}\label{rem: fixPointOnStack}
    In the case when it is possible to find a lift $\widetilde{\varphi}: \Gamma_F \to \check{T}'$ which is $\check{Y}$-generic, the fixed points are in bijection with lifts of $\varphi$ to $\check{T}'$, which is noncanonically parametrized by $\mathrm{Hom}(\Gamma_F, \check{\mu})$.
\end{rem}

Let $\varphi: \Gamma_F \to \check{T}(\kk)$ be an unramified $L$-parameter. Following \eqref{eq: localLfunction}, for $x_0 = (y,\widetilde{\varphi})$ a fixed point of $\Gamma_F$ on $\check{X}$ we define the \textit{local $L$-function attached to $x_0$ at a place $v$} as
\begin{equation}\label{eq: localLFunctionStack}
    L_v(\check{X}, x_0, \varphi,s) := \mathrm{gtr}\big(\widetilde{\varphi}(\mathrm{Fr}_v) \times q_v^{-s} \, | \, \widehat{\mathcal{O}}_{\check{Y}, y} \big) \in \kk[[q_v^{-s}]]
\end{equation}
and we define the \textit{normalized spectral period} by taking an Euler product and summing over fixed points with the same formula as \eqref{LXdef}, but we only consider the \textit{unramified} fixed points, from which we define
\begin{equation} \label{eqn: unramSpecPeriod}
    L_{\check{X}}^{\mathrm{ur}}(\varphi) := (\check{\rho} \circ \varphi)(\partial^{-1/2})\Delta^{\frac{\varepsilon- \mathrm{dim}(\check{X})}{4}}\sum_{x \text{ unramified}} \, \prod_v \, L_v(\check{X}, x, \varphi, \frac{1}{2}) 
\end{equation}

\begin{rem}
We justify the definition of $L_{\check{X}}^{\mathrm{ur}}$ briefly from the geometric perspective of BZSV's $L$-sheaves (in the de Rham context for expositional convenience). Let's take $\Sigma$ to be a complex algebraic curve only for the discussion of this remark, and consider the de Rham stack $\Sigma_{\mathrm{dR}}$ of $\Sigma$. One forms the diagram of mapping stacks
    \[\begin{tikzcd}
	{\mathrm{Loc}_{\check{T}'}^{\check{Y}}(\Sigma) = \mathrm{Map}(\Sigma_{\mathrm{dR}}, [\check{Y}/\check{T}'])} & {\mathrm{Loc}_{\check{T}'}(\Sigma)} \\
	{\mathrm{Loc}_{\check{T}}^{\check{X}}(\Sigma) = \mathrm{Map}(\Sigma_{\mathrm{dR}}, [\check{X}/\check{T}])} & {\mathrm{Loc}_{\check{T}}(\Sigma)}
	\arrow["{\pi_{\check{Y}}^{\mathrm{spec}}}", from=1-1, to=1-2]
	\arrow["\cong"', from=1-1, to=2-1]
	\arrow["{\check{\pi}}", from=1-2, to=2-2]
	\arrow["{\pi_{\check{X}}^{\mathrm{spec}}}"', from=2-1, to=2-2]
\end{tikzcd}\]
The sheaf $\pi^{\mathrm{spec}}_{\check{X},*} \omega_{\mathrm{Loc}_{\check{T}}^{\check{X}}}$ ought to be the $L$-sheaf attached to $\check{X}$, geometrizing the spectral period, where the symbol $\omega_Z$ denotes the dualizing complex on $Z$. Our sought after definition of $L_{\check{X}}^{\mathrm{ur}}$ then comes from pairing with the skyscraper sheaf at a $\check{T}$-local system: for $\delta_\varphi$ the skyskraper sheaf at a local system $\varphi$, the spectral $\check{X}$-period of $\varphi$ is geometrized by the Hom-space
$$\mathrm{Hom}\left(\delta_\varphi, \pi^{\mathrm{spec}}_{\check{X},*} \omega_{\mathrm{Loc}_{\check{T}}^{\check{X}}}\right) = \mathrm{Hom}\left(\delta_\varphi,(\check{\pi}_* \circ \pi^{\mathrm{spec}}_{\check{Y},*})\omega_{\mathrm{Loc}_{\check{T}'}^{\check{Y}}}\right) = \mathrm{Hom}\left(\check{\pi}^*\delta_\varphi,\pi^{\mathrm{spec}}_{\check{Y},*}\omega_{\mathrm{Loc}_{\check{T}'}^{\check{Y}}}\right)$$
Looking at the right hand side of this equality, we see that we ought to look for lifts $\widetilde{\varphi}$ of $\varphi$ to $\check{T}$-local systems, and we sum over the $\check{Y}$-spectral period of such lifts. This geometrizes the definition given by Equation \eqref{eqn: unramSpecPeriod}.
\end{rem}

To make a complete definition of weak numerical duality, we also need to give an automorphic period attached to the types of Deligne--Mumford stack we consider. 

\begin{defn} \label{defn: automorphicPeriodDM}
    Let $\check{X} = \check{Y} \times^{\check{T}'} \check{T}$ be the $\check{T}$-toric Deligne--Mumford stack as above, equipped with a rational $\check{T}$-character $\check{\rho}$ and a cocharacter $\eta: \Ggr \to \check{T}$. We define the \textit{regularized theta series} of $\check{X}$ as the function on $[\check{T}']$ given by
    $$\mathring{\theta}_{\check{X}}(t) := \sum_{y \in \mathring{\check{Y}}(F)} \, (t, \partial^{1/2}) \star \Phi(y)$$
    where $\Phi$ is our standard choice of basic Schwartz function on $Y(\mathbb{A})$.
    
    For a character $\chi$ of $[\check{T}]$, we define the \textit{regularized automorphic $\check{X}$-period} of $\chi$ as
    $$\mathring{P}_{\check{X}}(\chi) = \int_{[\check{T}']} \, \mathring{\theta}_{\check{X}}(t) \chi(\check{\pi}(t)) \, dt$$
\end{defn}

\begin{rem}
    We justify also the definition of $\mathring{\theta}_{\check{X}}$ and $\mathring{P}_{\check{X}}$ briefly from the geometric perspective of BZSV's period sheaves. One considers the diagram of mapping stacks
   \[\begin{tikzcd}
	{\mathrm{Bun}_{\check{T}'}^{\check{Y}}(\Sigma) = \mathrm{Map}(\Sigma, [\check{Y}/\check{T}'])} & {\mathrm{Bun}_{\check{T}'}(\Sigma)} \\
	{\mathrm{Bun}_{\check{T}}^{\check{X}}(\Sigma) = \mathrm{Map}(\Sigma, [\check{X}/\check{T}])} & {\mathrm{Bun}_{\check{T}}(\Sigma)}
	\arrow["{\pi_{\check{Y}}^{\mathrm{aut}}}", from=1-1, to=1-2]
	\arrow["\cong"', from=1-1, to=2-1]
	\arrow["{\check{\pi}}", from=1-2, to=2-2]
	\arrow["{\pi_{\check{X}}^{\mathrm{aut}}}"', from=2-1, to=2-2]
\end{tikzcd}\]
    The sheaf $\pi_{\check{X}, !}^{\mathrm{aut}}\underline{k}$ ought to be the period sheaf attached to $\check{X}$, geometrizing $\mathring{\theta}_{\check{X}}$\footnote{We have ignored two details in this remark: 1) there ought to be a $K^{1/2}$-twist since we are considering the normalized period, and 2) the regularized theta series is rather geometrized by the constant sheaf on a certain dense substack of $\mathrm{Bun}_{\check{T}}^{\check{X}}(\Sigma)$. These details are not critically important for the purposes of this remark.}. The automorphic $\check{X}$-period evaluated on a character $\chi$ of $\mathrm{Bun}_{\check{T}}(\Sigma)$ is geometrized by the Hom-space $\mathrm{Hom}(\chi, \pi_{\check{X}, !}^{\mathrm{aut}}\underline{k})$, which we can rewrite as
    $$\mathrm{Hom}(\chi, \pi_{\check{X}, !}^{\mathrm{aut}}\underline{k}) = \mathrm{Hom}(\chi, \check{\pi}_*\pi_{\check{Y},!}^{\mathrm{aut}}\underline{k}) = \mathrm{Hom}(\check{\pi}^*\chi, \pi_{\check{Y},!}^{\mathrm{aut}}\underline{k})$$
    the latter of which geometrizes the integral written in Definition \ref{defn: automorphicPeriodDM}. 
\end{rem}

Finally, we discuss the most interesting arithmetic aspect of extending the notion of weak duality to this setting: \textit{the arithmetic orbit structure of $X$ should match the finite stabilizers of $\check{X}$.} To make this precise, note that we have the Kummer exact sequence
$$0 \longrightarrow \mu(F) \longrightarrow  T(F) \longrightarrow \mathring{X}(F) \longrightarrow H^1(\mathrm{Gal}_F, \mu) \longrightarrow H^1(\mathrm{Gal}_F, T) = 0$$
where $\mathrm{Gal}_F$ denotes the absolute Galois group of $F$. Assume from now on that
\begin{equation} \label{assump: rootsOfUnity}
    F \text{ has all the roots of unity of order dividing the exponents of } \mu,
\end{equation} 
then we will have an identification 
\begin{equation} \label{eqn: canonicalBijection}
    T(F) \text{-orbits on } \mathring{X}(F) \longleftrightarrow H^1(\mathrm{Gal}_F, \mu) = \mathrm{Hom}(\mathrm{Gal}_F, \mu)
\end{equation}
We say that a $T(F)$-orbit on $\mathring{X}(F)$ is \textit{unramified} if its corresponding Galois cohomology class in $H^1(\mathrm{Gal}_F, \mu)$ is unramified. 

On the other hand, for a $\check{T}$-valued $L$-parameter $\varphi: \Gamma_F \to \check{T}$ which posseses a $\check{Y}$-generic lift $\widetilde{\varphi}: \Gamma_F \to \check{T}'$, we see by Remark \ref{rem: fixPointOnStack} that the unramified fixed points of $\varphi$ on $\check{X}$ are indexed by unramified lifts of $\varphi$. This forms a torsor under $\mathrm{Hom}(\Gamma_F, \check{\mu})$, which is noncanonically isomorphic to $\mathrm{Hom}(\Gamma_F, \mu)$ since $\mu \cong \check{\mu}$ over $F$ by our assumption \eqref{assump: rootsOfUnity}. Thus, there is a noncanonical bijection
\begin{equation} \label{eqn: noncanonicalBij}
    \text{unramified }T(F)\text{-orbits on } \mathring{X}(F) \longleftrightarrow \mathrm{Hom}(\Gamma_F, \check{\mu})
\end{equation}
where, we recall per our conventions that $\Gamma_F$ is the unramified Weil group of $F$.

\begin{defn}
    We define the \textit{(regularized) unramified $X$-theta series} as
    $$\mathring{\theta}_X^{\mathrm{ur}}(t) = \sum_{x \in \mathring{X}(F) \text{ unramified}} \, (t, \partial^{1/2}) \star \Phi(x)$$
    and we define the \textit{(regularized) unramified automorphic $X$-period} of a character $\chi$ of $[T]$ as
    \begin{equation}
        \mathring{P}_X^{\mathrm{ur}}(\chi) = \int_{[T]} \, \chi(t) \mathring{\theta}_X^{\mathrm{ur}}(t) \, dt
    \end{equation}
\end{defn}

\begin{rem} 
    In the definition of $\mathring{\theta}_X^{\mathrm{ur}}$ and thus in $\mathring{P}_X^{\mathrm{ur}}$ we were careful in only taking those rational points $x \in \mathring{X}(F)$ that are unramified. However, if $x \in \mathring{X}(F)$ is an arbitrary (not necessarily unramified) orbit, then it contributes via equation \eqref{eq: arithmeticOrbits} another Eulerian integral. We see that formally, the noncanonical bijection \eqref{eqn: noncanonicalBij} can be extended to one between all $T(F)$-orbits on $\mathring{X}(F)$ and $\mathrm{Hom}(\mathrm{Gal}_F, \check{\mu})$, hence the collection of Eulerian integrals contributed by all orbits ought to correspond to the spectral period contributed by all lifts of the $L$-parameter, with an appropriate choice of ramified Schwartz function at the ramified places.
\end{rem}

\begin{defn}\label{defn: weakDualityStack}
    Let $T, \check{T}$ be split Langlands dual tori. Let $\mu \subset T$ be a finite subgroup scheme, from which we form the quotient torus $T' = T/\mu$ whose dual torus we denote by $\check{T}'$. Let $(T',Y)$ and $(\check{T}', \check{Y})$ be a graded toric dual pair, from which we construct the graded toric dual $(T,X)$ and $(\check{T}, \check{X})$, as above. Then we say that $(T, X)$ and $(\check{T}, \check{X})$ are \textit{weakly numerically dual with discrepancy $(a, \check{a}) \in \ZZ^2$} if for 
    \begin{itemize}
        \item any finite field $\F$ of size $q$ prime to the order of $\mu$ and containing all roots of unity with order dividing the exponents of $\mu$, and
        \item any curve $\Sigma$ over $\F$ with function field $F$,
    \end{itemize}  
   we have 
   \begin{itemize}
       \item for any unramified character $\chi$ of $[T]$, the identity
       $$\mathring{P}_X^{\mathrm{ur}}(\chi) = \frac{\Delta^{a/4}}{\big[[T']:\pi([T])\big]} \, L^{\mathrm{ur}}_{\check{X}}(\varphi_\chi)$$
       where $\varphi_\chi: \Gamma_F \to \check{T}$ is the $L$-parameter of $\chi$, and
       \item for any unramified character $\chi$ of $[\check{T}]$, the identity
       $$\mathring{P}_{\check{X}}(\chi) = \Delta^{\check{a}/4} \, L_X(\varphi_\chi)$$
       where $\varphi_\chi$ is the $L$-parameter of $\chi$.
   \end{itemize}
\end{defn}

Now we come to the main theorem of this section.

\begin{thm} \label{thm: DMTheorem}
    Let $(T,X)$ and $(\check{T}, \check{X})$ be a graded toric dual pair induced from a finite subgroup $\mu \subset T$ and a toric dual pair $(T', Y)$ and $(\check{T}', \check{Y})$ as above. Then $(T,X)$ and $(\check{T}, \check{X})$ are weakly numerically dual, with the same discrepancy as the weakly numerically dual pair $(T', Y)$ and $(\check{T}', \check{Y})$. 
\end{thm}
\begin{proof}
    We first start with $(\check{T}, \check{X})$ on the automorphic side, and we write $(b,\check{b})$ for the discrepancy of the weakly numerically dual pair $(T', Y)$ and $(\check{T}', \check{Y})$. For an unramified character $\chi$ of $[\check{T}]$ with $L$-parameter $\varphi: \Gamma_F \to T(\kk)$, we see that the automorphic $\check{X}$-period of $\chi$ on $\check{T}$ is just the automorphic $\check{Y}$-period of $\chi \circ \check{\pi}$ on $\check{T}'$, so we have
    $$\mathring{P}_{\check{X}}(\chi) = \mathring{P}_{\check{Y}}(\chi \circ \check{\pi}) = \Delta^{\check{a}/4} \, L_Y(\pi \circ \varphi) = \Delta^{\check{a}/4} \, L_X(\varphi)$$
    as one half of the duality, with $\check{b} = \check{a}$. 
    
    Now we place $(T,X)$ on the automorphic side. For an unramified character $\chi$ on $[T]$ with $L$-parameter $\varphi: \Gamma_F \to \check{T}(\kk)$, we calculate the automorphic $X$-period 
    \begin{align*}
        &\mathring{P}_X^{\mathrm{ur}}(\chi) = \int_{[T]} \, \chi(t) \mathring{\theta}_X^{\mathrm{ur}}(t) \, dt
    \end{align*}
in terms of $\widetilde{\chi}$, an arbitrary (unramified) extension of $\chi$ to $[T']$ so that $\chi = \widetilde{\chi} \circ \pi$. Such a lift always exists since $\chi$ is unramified so we can view it as a character on $[T] = T(F) \backslash T(\mathbb{A})/T(\mathcal{O})$, and the homomorphism
$$\pi: T(F) \backslash T(\mathbb{A})/T(\mathcal{O}) \to T'(F) \backslash T'(\mathbb{A})/T'(\mathcal{O})$$
is injective; indeed, this follows from the fact that if a rational element $x \in F^\times$ has an $n$th root everywhere locally, then $x$ is itself a rational $n$th power. 

Now we make the change of variables $\tau = \pi(t)$ and apply Fourier expansion along the subgroup $\pi([T]) \subset [T']$ to obtain
    \begin{align*}
        &\mathring{P}_X^{\mathrm{ur}}(\chi) = \int_{\pi([T])} \, \widetilde{\chi}(\tau) \mathring{\theta}_X^{\mathrm{ur}}(\tau) d\tau\\
        &= \frac{1}{\big[[T']:\pi([T])\big]}\sum_{\varepsilon: \pi([T]) \backslash [T'] \to \CC^\times} \, \int_{[T']} \, \widetilde{\chi}(\tau)\varepsilon(\tau)\mathring{\theta}_Y(\tau) d\tau
    \end{align*} 
    Note that since $\widetilde{\chi}\mathring{\theta}_Y$ is $T'(\widehat{\mathcal{O}})$-invariant, only those summands corresponding to unramified characters $\varepsilon$ of $[T']$ will be nonzero. Furthermore, the set of characters 
    $$\big\{ \, \widetilde{\chi}\varepsilon: \varepsilon: \pi([T]) \backslash [T'] \to \CC^\times \text{ unramified }\big\}$$
    is exactly the set of unramified lifts of $\chi$ to $[T']$. The $L$-parameters of this set of characters is then the set of unramified lifts of $\varphi$ to $\check{T}'$, and we can also rewrite the previous line as
    \begin{align*}
        \mathring{P}_X^{\mathrm{ur}}(\chi) &= \frac{1}{\big[[T']:\pi([T])\big]}\sum_{\widetilde{\chi} \, : \, \chi = \widetilde{\chi} \circ \pi} \, \mathring{P}_Y(\widetilde{\chi})\\
        &= \frac{\Delta^{a/4}}{\big[[T']:\pi([T])\big]}\sum_{\widetilde{\varphi} \, : \, \varphi = \check{\pi} \circ \widetilde{\varphi}} \, L_{\check{Y}}(\widetilde{\varphi}) = \frac{\Delta^{a/4}}{\big[[T']:\pi([T])\big]}\, L_{\check{X}}^{\mathrm{ur}}(\varphi),
    \end{align*}
    so that $b = a$.
\end{proof}

\begin{rem}
    In considering only unramified characters we overlook a pleasing feature of the above calculation, resembling $\varepsilon$-dichotomy phenomena, when $(T,X)$ is placed on the automorphic side and $(\check{T}, \check{X})$ on the spectral side. Suppose $\chi$ is a character of $[T]$ ramified at some finite set of places $S$, with $U_v = \mathrm{Ker}(\chi_v) \subset T_v$. Then it is no longer the case that the adelic quotient of the finite group $\mu$ is trivial: indeed, instead we have the finite set
    $$\mu(F) \backslash \mu(\mathbb{A})/ \mu(\widehat{\mathcal{O}}_{v \not \in S}) = \mu(F) \backslash \prod_{v \in S} \, \mu(F_v) =: E.$$
    In unfolding the integral $P_X(\chi)$ exactly as in \eqref{eq: orbitDecomp}, we let $\{x\}$ be a set of $T(F)$-orbit representatives on $X(F)$, and obtain
    \begin{equation}
        P_X(\chi) = \chi(\partial^{-1/2})\sum_{x \in X(F)/T(F)} \, \int_{\mu(\mathbb{A}) \backslash T(\mathbb{A})} \, \Phi(xt)\chi|\eta|^{1/2}(t) \, \sum_{\varepsilon \in E} \, \chi(\varepsilon) \, dt
    \end{equation}
    so we see that a necessary condition for the nonvanishing of $P_X(\chi)$ is the triviality of $\chi|_E$.

    Dually, we see that the stack $\check{X} = \check{Y} \times^{\check{T}'} \check{T}$ distinguishes $L$-parameters that factor through the finite cover $\check{T}' \to \check{T}$. This condition is equivalent to $\chi$ adimitting a lift to $[T']$, which is in turn equivalent to $\chi|_E = \mathbf{1}$, a condition trivially satisfied by unramified characters. 
\end{rem}

One can imagine a generalization of Theorem \ref{thm: DMTheorem} involving central isogenies of general reductive groups: suppose $G', \check{G}'$ is a pair of Langlands dual reductive groups, and
$$1 \longrightarrow \mu \longrightarrow G \longrightarrow G' \longrightarrow 1$$
is a central isogeny with kernel $\mu \subset G$, presenting $G$ as a covering group of $G'$. Then we have a dual sequence
$$1 \longrightarrow \check{\mu} \longrightarrow \check{G}' \longrightarrow \check{G} \longrightarrow 1$$
which presents $\check{G}'$, the Langlands dual group of $G$, as a cover of $\check{G}$. If $(G', Y)$ and $(\check{G}', \check{Y})$ is a weakly numerically dual pair, then we can consider the pair
$$(G, X = Y) \text{ and } (\check{G}, \check{X} = \check{Y} \times^{\check{G}'} \check{G})$$
where $(G,X)$ has disconnected generic stabilizers, and $\check{X}$ is a Deligne--Mumford stack. We expect the pair $(G,X)$ and $(\check{G}, \check{X})$ to be weakly numerically dual, following the same calculations as done in Theorem \ref{thm: DMTheorem}. A much more interesting situation is when the covering $G \to G'$ is \textit{metaplectic}; we will return to this question via some examples elsewhere.

\appendix
\section{The work of Batyrev--Tschinkel on Manin's Conjecture} \label{appendix: BT}

We thank Philippe Michel for bringing to our attention the fact that the spectral periods we consider had appeared in the literature in another guise, in the work of Batyrev--Tschinkel on Manin's conjecture for toric varieties \cite{Batyrev-Tschinkel}. Combining the two perspectives suggests generalizations of our constructions to \textit{arbitrary} toric varieties, i.e. not only affine ones, and definitions of archimedean local integrals. We hope to return to these questions in another note in the near future, while for now we content ourselves with explaining the relationship between the two calculations. 

For a $T$-toric variety $X$ defined by a fan $\Sigma \subset X_*(T)_\RR$ and a $\CC$-valued, piecewise linear function $\varphi$ on $\Sigma$, one defines a height function
$$H( \, \cdot \,, \varphi): T(F) \longrightarrow \RR$$
$$H(x, \varphi) := \prod_{v < \infty} \, q_v^{-\varphi(\overline{x})} \prod_{v | \infty} e^{\varphi(\overline{x})}$$
where $\overline{x}$ is the image of $x$ inside $T(F_v)/T(\mathcal{O}_v) \cong X_*(T)$. Essentially the choice of $\varphi$ encodes the choice of a metrized line bundle with which one defines a notion of height for the rational points inside the dense $T$-orbit inside $X$. Note also that this is analogous to our choice of eigenform $\eta$ on $X$, whose fan consists of just one cone of maximal dimension. Furthermore, the height function plays the same role as our Schwartz function just as a Siegel section plays the same role as the Schwartz function in the formation of an incomplete Eisenstein series.

In the context of Manin's conjecture, one is interested in the asymptotic behavior of the number of rational points of bounded height: for $B \geq 0$, one considers
$$N(X, \varphi, B) := \big|\big\{x \in T(F) \text{ such that } H(x,\varphi) \leq B \big\}\big|$$
and the conjecture gives an asymptotic for $N(X, \varphi,B)$ as $B \to \infty$ (see pg. 5 of \textit{loc. cit}). Applying a Tauberian theorem, one turns to analyzing the analytic behavior of the \textit{height zeta function}
$$Z(\varphi) := \sum_{x \in T(K)} \, H(x,\varphi)$$
whose argument lies in the space of piecewise linear functions on $\Sigma$ (when it converges). 

Viewing $H( \, \cdot \, , \varphi)$ as a function on the adelic torus $T(\mathbb{A})$ and noting that $Z(\varphi)$ can be rewritten using Poisson summation, one is led to compute the Fourier transform of $H( \, \cdot \, , \varphi)$: for a character $\chi$ of $[T]$, one then considers the integral
$$\widehat{H}(\chi,\varphi) := \int_{[T]} \, \chi(t) H(t,\varphi) \, dt$$
which is more or less our period integral $P_X(\chi)$, and a direct application of Poisson summation gives
$$Z(\varphi) \sim \int_{\chi \in \widehat{[T]}} \, \widehat{H}(\chi, \varphi) \, d\chi$$
where $\sim$ means up to an explicit rational factor, and $\widehat{[T]}$ is the space of unramified unitary characters of $[T]$. Here we only need to consider unramified characters since, $H$ being $T(\mathcal{O}_v)$-invariant for all $v$, we see that $\widehat{H}(\chi,\varphi) = 0$ unless $\chi$ is unramified. A careful analysis of the $\widehat{H}(\chi)$ yielded precise expressions of the leading terms of $Z$ at its poles. 

One expects naturally then that general nonlinear $L$-functions $L_{\check{X}}$ should help us approach Manin's conjecture on $X$; this seems to be a fruitful way to understand recent work on Manin's conjecture on spherical Fano varieties \cite{BBDG}.

\end{document}